\documentclass[12pt,a4paper]{amsart}
\usepackage{latexsym,amsfonts,amsthm,amsmath,mathrsfs,amssymb}
\usepackage[british]{babel}
\usepackage[dvips]{color}
\usepackage[utf8]{inputenc}
\usepackage[T1]{fontenc}
\usepackage[dvips,final]{graphics}
\usepackage{xcolor}
\usepackage{mathrsfs}
\usepackage{mathtools} 
\usepackage{breqn}
\usepackage{epstopdf}
\usepackage{verbatim}
\usepackage{amscd}
\usepackage{hyperref}
\usepackage[english,capitalize]{cleveref}
\usepackage{aliascnt}

\usepackage{todonotes}
\makeatletter
\def\newaliasedtheorem#1[#2]#3{
	\newaliascnt{#1@alt}{#2}
	\newtheorem{#1}[#1@alt]{#3}
	\expandafter\newcommand\csname #1@altname\endcsname{#3}
}
\makeatother

\theoremstyle{plain}

\newtheorem{theorem}{Theorem}[section]
\newaliasedtheorem{prop}[theorem]{Proposition}
\newaliasedtheorem{lem}[theorem]{Lemma}
\newaliasedtheorem{coroll}[theorem]{Corollary}

\theoremstyle{definition}
\newaliasedtheorem{defi}[theorem]{Definition}
\newaliasedtheorem{quest}[theorem]{Question}
\newaliasedtheorem{fact}[theorem]{Fact}

\theoremstyle{remark}
\newaliasedtheorem{rem}[theorem]{Remark}
\newaliasedtheorem{exa}[theorem]{Example}

\newcommand{\R}{\mathbb{R}}

\newcommand{\C}{\mathbb{C}}

\let\altphi\phi
\let\phi\varphi
\let\varphi\altphi
\let\altphi\undefined

\newcommand{\average}{{\mathchoice {\kern1ex\vcenter{\hrule height.4pt
width 6pt
depth0pt} \kern-9.7pt} {\kern1ex\vcenter{\hrule height.4pt width 4.3pt
depth0pt}
\kern-7pt} {} {} }}

\address{\textsc{Daniela Di Donato}: 
Dipartimento di Ingegneria Industriale e Scienze Matematiche, Via Brecce Bianche, 12 60131 Ancona, Universit\'a Politecnica delle Marche.}
\email{daniela.didonato@unitn.it}

\title{Intrinsically quasi-isometric sections in metric spaces}

\date{\today}

\author{ Daniela Di Donato }

\subjclass[]{ 
26A16  % Lipschitz (Hlder) classes
51F30 % Lipschitz and coarse geometry in metric space
46B04 % isometric theory of banach space
54E35 % metric spaces
}
\keywords{Large scale geometry, Quasi-isometric graphs, vector space, Ahlfors-David regularity, Metric spaces}

\begin{document}

\begin{abstract} This note is a contribution to large scale geometry. More precisely, we introduce the intrinsically quasi-isometric sections in metric spaces and we  investigate their properties: the Ahlfors-David regularity in large scale; following Cheeger theory, it is possible to define suitable sets in order to obtain convexity and being a vector space over $\R$ or $\C$ for these sections; yet, following Cheeger's idea, we give an equivalence relation for this class of sections. Throughout the paper, we use basic mathematical tools.	\end{abstract}

\maketitle 
\tableofcontents

\section{Introduction}
This paper is a contribution to Large Scale Geometry,  also known as coarse geometry. Large Scale Geometry is the study of geometric objects viewed from a great distance. It plays an important role in geometric group theory, algebraic K-theory, non-commutative geometry, and related areas of analysis. The reader can see \cite{NG12, BW97, Roe2003LecturesOC}.

In particular, we focus our attention on the concept of quasi-isometric graphs in metric spaces.  In general, quasi-isometric maps \cite{G87, LD17, S95} are large scale version of biLipschitz maps and the latter have been studied by Le Donne and the author  in \cite{DDLD21}. More precisely,  Le Donne and the author  \cite{DDLD21} give a 'different' notion of Lipschitz graph starting from two simple facts:
\begin{enumerate}
\item Franchi, Serapioni and Serra Cassano  \cite{FSSC, FSSC03, MR2032504} introduced and studied the class of intrinsically Lipschitz maps in subRiemannian Carnot groups in order to establish a good notion of rectifiability sets in the context of subRiemannian Carnot groups \cite{ABB, BLU, CDPT} after the negative result in \cite{AmbrosioKirchheimRect}. 
\item we consider graphs instead of maps.
\end{enumerate}

 In our context we consider a section $\phi$ of $\pi$ (i.e., $\pi \circ \phi = id$) such that $\pi:X \to Y$ produces a foliation for $X,$ i.e., $X= \coprod \pi ^{-1} (y)$ and the Lipschitz property of $\phi$ consists to ask that the distance between two points is comparable with the distance between a point and a fiber.  Following this idea, it is natural to study other notions like intrinsically H\"older \cite{D22.1} and  quasi-symmetric \cite{D22.2} sections where in the right side term we consider the distance with a fiber instead of a point. In this paper, we introduce the intrinsically quasi-isometric sections as follows: we have a metric space $X$, a topological space $Y$, and a 
quotient map $\pi:X\to Y$, meaning
continuous, open, and surjective.
%In some situations we might assume that $Y$ is metrizable so that $\pi$ becomes a Lipschitz quotient,  e.g. a submetry. normal
The standard example for us is when $X$ is a metric Lie group $G$ (meaning that the Lie group $G$ is equipped with a left-invariant distance that induces the manifold topology), for example a subRiemannian Carnot group, %$N\lhd G$ 
and $Y$ if the space of left cosets $G/H$, where 
$H<G$ is a  closed subgroup and $\pi:G\to G/H$ is the projection modulo $H$, $g\mapsto gH$.

\begin{defi}\label{def_ILS}We say that a map $\phi:Y\to X$ is an {\em intrinsically $(L,M)$-roughly quasi-isometric section of $\pi$} or, simply, intrinsically $(L,M)$-quasi-isometric section of $\pi$,  with $L\geq 1$ and $M\geq 0$, if
%Let $(X,d)$ be a metric space and let $Y$ be a topological space. We say that a map $\phi :Y \to X$ is a section of a quotient map $\pi :X \to Y$ if
\begin{equation}
\pi \circ \phi =\mbox{id}_Y,
\end{equation}
and
\begin{equation}
d(\phi (y_1), \phi (y_2)) \leq L d(\phi (y_1), \pi ^{-1} (y_2)) +M, \quad \mbox{for all } y_1, y_2 \in Y.
\end{equation}
Here $d$ denotes the distance on $X$, and, as usual, for a subset $A\subset X$ and a point $x\in X$, we have
$d(x,A):=\inf\{d(x,a):a\in A\}$.
\end{defi}

If $L=1,$ then $\phi$ is called an $M$-roughly isometric; on the other hand if $M=0$ the intrinsically $L$-roughly quasi-isometric sections are  intrinsically  $L$-Lipschitz  sections studied in \cite{DDLD21}. Moreover, we underline that, when $M=0$ and $
 \pi$ is a Lipschitz quotient or submetry \cite{MR1736929, Berestovski}, the results trivialize, since in this case being intrinsically Lipschitz  is equivalent to biLipschitz embedding, see Proposition 2.4 in \cite{DDLD21}.

It easy to see that an intrinsically $(L,M)$-quasi-isometric section need not be continuous either. Moreover, by the simply fact that $\phi (y) \in \pi^{-1} (y),$ we have that
\begin{equation*}
\frac 1 L d(\phi (y_1), \pi ^{-1} (y_2)) -M\leq  d(\phi (y_1), \pi ^{-1} (y_2)) \leq d(\phi (y_1), \phi (y_2)), 
\end{equation*}
 for any $y_1,y_2 \in Y$ and so we get the left part of  classical  quasi-isometric definition  is trivial, i.e.,
 \begin{equation*}
 \frac 1 L d(\phi (y_1), \pi ^{-1} (y_2)) -M \leq d(\phi (y_1), \phi (y_2)) \leq L d(\phi (y_1), \pi ^{-1} (y_2)) +M, 
\end{equation*} 
  for any $y_1,y_2 \in Y.$

 \medskip

The main result of this paper is the following.
   \begin{theorem}[Ahlfors-David regularity]\label{thm2}
   Let $\pi :X \to Y$ be a quotient map between a metric space $X$ and a topological space $Y$ such that there is a measure $\mu$ on $Y$ such that for every $r_0>0$ and every $x,x' \in X$ with $\pi (x)=\pi(x')$  there is $C>0$ such that
      \begin{equation}\label{Ahlfors27ott.112}
\mu (\pi (B(x,r))) \leq C \mu (\pi (B(x',r))),
\end{equation}
for  every $r>r_0.$
   
We also assume that $\phi :Y \to X$ is an intrinsically $(L,M)$-quasi-isometric section of $\pi$ such that $\phi (Y)$ is, for large scale, $Q$-Ahlfors-David regular with respect to  the measure $\phi_* \mu$, with $Q\in (0,\infty)$. 

  Then, for every intrinsically  $(L,M)$-quasi-isometric section $\psi :Y \to X,$  the set $ \psi (Y) $ is such that  %there are $c_3,c_4 >0$ such that
   \begin{equation}\label{AhlforsNEW127}
 (r-M)^Q \lesssim \psi _* \mu \big( B(\psi(y),r) \cap \psi (Y)\big)  \lesssim (r+M)^Q,
\end{equation}
  for every $r > r_0.$

  %, for large scale, $Q$-Ahlfors regular with respect to  the measure $\psi_* \mu.$ 

    \end{theorem}
 Namely, in Theorem~\ref{thm2}  Ahlfors-David $Q$-regularity means that    the measure $\phi_* \mu$ is such that  for each point $x\in \phi (Y)$ there exist $r_0>0$ and $C>0$ so that
  \begin{equation}\label{Ahlfors_IN_N}
 C^{-1}r^Q\leq  \phi_* \mu \big( B(x,r) \cap \phi (Y)\big) \leq C r^Q, \qquad \text{ for all }r>r_0.
\end{equation}

Finally, thanks to the seminal paper of Cheeger \cite{C99} (see also \cite{K04, KM16}) we give an  equivalent property of quasi-isometric sections which we will be the basic point in order to obtain the following results:
\begin{enumerate}
\item Proposition \ref{propLeibnitz formula for slope}: a suitable set of this class of sections is convex one.
\item Theorem \ref{theorem24} and \ref{theorem30}: a suitable set of this class of sections is a vector space over $\R$ or $\C.$
\item Theorem \ref{theorem26aprile} gives an equivalence relation of this class of sections.
\end{enumerate}
Regarding being vector spaces over  $\R$ or $\C,$ we present two different approaches: in Theorem \ref{theorem24} we used to being quasi-isometric section with respect to another section; on the other hand, in Theorem \ref{theorem30} we used the finite distance between two points belong to the same fiber.

 \begin{prop}\label{linkintrinsicocneelip}
   Let $X  $ be a metric space, $Y$ a topological space, $\pi :X \to Y$   a  quotient map, $L\geq 1$ and $M\geq 0$.
Assume that   every point $x\in X$ is contained in the image of an intrinsic $(L,M)$-quasi-isometric section $\psi_x$ for $\pi$.
 Then for every section $\phi :Y\to X$ of $\pi$ the following are equivalent:
   \begin{enumerate}
\item for all $x\in\phi(Y)$ the section $\phi $ is intrinsically $(L_1,M_1 )$-quasi-isometric with respect to  $\psi_x$ at   $x;$
\item  the section $\phi $  is intrinsically $(L_2, M_2)$-quasi-isometric.
\end{enumerate}
Moreover, if $\psi$ is intrinsically Lipschitz $($i.e. $M=0)$, then $M_1=M_2.$
   \end{prop}
 
 {\bf Acknowledgements.}  We would like to thank Giorgio Stefani for the reference \cite{KM16}.

\section{Equivalent definitions for intrinsically quasi-isometric  sections}
\label{sec:equiv_def}

\begin{defi}[Intrinsic quasi-isometric section]\label{Intrinsic Lipschitz section}
Let $(X,d)$ be a metric space and let $Y$ be a topological space. We say that a map $\phi :Y \to X$ is a {\em section} of a quotient map $\pi :X \to Y$ if
\begin{equation*}
\pi \circ \phi =\mbox{id}_Y.
\end{equation*}
Moreover, we say that $\phi$ is an {\em intrinsically $(L,M)$quasi-isometric section} with constants $L\geq 1$ and $M\geq 0$ if in addition
\begin{equation*}
d(\phi (y_1), \phi (y_2)) \leq L d(\phi (y_1), \pi ^{-1} (y_2)) +M, \quad \mbox{for all } y_1, y_2 \in Y.
\end{equation*}

Equivalently, we are requesting that  that
\begin{equation*}
d(x_1, x_2) \leq L d(x_1, \pi ^{-1} (\pi (x_2)))+M, \quad \mbox{for all } x_1,x_2 \in \phi (Y) .
\end{equation*}
\end{defi}

We further rephrase the definition as saying that $\phi(Y)$, which we call the {\em graph} of $\phi$, avoids some particular sets (which depend on $L, M$ and $\phi$ itself):

%Given $x\in X$ and $L\geq 0$ we define the following two sets:
%\begin{eqnarray*}
%R_{x,L} & = & \{ x'\in X \;|\;   d(x', \pi ^{-1} (\pi (x))) < L d(x', x)\}, \\ 
%R'_{x,L} & = & \{ x'\in X \;|\;   d(\pi ^{-1} (\pi (x')), \pi ^{-1} (\pi (x))) < L d(x', x)\}.\\
%R''_{x,L} & = & \{ x'\in X \;|\;   d(\pi (x'), \pi (x)) < L d(x', x)\},\\
%\end{eqnarray*}
%where the last one makes sense only if $Y$ is metrized.

\begin{prop}\label{propo_ovvia} Let $\pi :X \to Y$  be a  quotient map between a metric space and a topological space, $\phi: Y\to X$ be a section of $\pi$, $L\geq 1$ and $M\geq 0$.
Then $\phi$ is intrinsically $(L,M)$-quasi-isometric if and only if
\begin{equation*}
 \phi  (Y) \cap  R_{x,L} = \emptyset , \quad \mbox{for all } x \in \phi (Y),
\end{equation*}
where $$R_{x,L} := \left\{ x'\in X \;|\;   L d(x', \pi ^{-1} (\pi (x))) +M<   d(x', x)\right\}.$$

%(ii) Assuming that $Y$ is metrizable so that $\pi$ is a $k$-Lipschitz quotient, then 

%(ii) If $\phi$ is $L$-intrinsically Lipschitz, then %since $R_{x,L} \subset R'_{x,L}, x' \ne R'_{x,L} then $Ld(x,x') \leq d(\pi ^{-1} (\pi (x')), \pi ^{-1} (\pi (x))) \leq d(x', \pi ^{-1} (\pi (x)))$
%\begin{equation*}\label{coni_LipQuo1} 
% \phi  (Y) \cap  R'_{x,L} = \emptyset , \quad \mbox{for all } x \in \phi (Y).
%\end{equation*}
%and
%\begin{equation}\label{coni_LipQuo1} Daniela scrivere con R''
%\end{equation}
%
%(ii$_1$) Vice versa, if either \eqref{coni_LipQuo1}  or \eqref{coni_LipQuo2}  hold, then $\phi$ is $kL$-intrinsically Lipschitz.
\end{prop}

Proposition \ref{propo_ovvia} is a triviality, still its purpose is to stress the analogy with the intrinsically Lipschitz sections theory introduced in \cite{DDLD21} when $M=0$ and, in particular, the sets $R_{x,L}$ are the intrinsic cones  in the sense of  Franchi, Serapioni and Serra Cassano  considered in subRiemannian Carnot groups when $M=0$. The reader can see \cite{D22.3} for a suitable notion of intrinsic cones in metric groups.
%{\bf METTER QUI CHE SE PI LIP QUOT ALLORA LIP EMBED}

\subsection{Intrinsic  quasi-isometric  with respect to  families of sections}
In this section we continue to fix a  quotient map $\pi :X \to Y$ between a metric space $X$ and a topological space $Y$.
%\begin{defi}[Intrinsic cone]\label{defiIntrCONEnew}
%
% After fixed $q\in X$, $L\geq 1$, and  a section $\phi_0 :Y\to X$ for $\pi$, we define 
%\begin{equation*}\label{intrcones}
% C_{q,L}^{\phi_0} := \{p\in X \,:\,  d(p, \phi_0 (\pi (p))) > L d(q, \phi_0 (\pi (p)))  \}.
%\end{equation*}
%\end{defi}

\begin{defi}[Intrinsic quasi-isometric  with respect to  a section]\label{defwrtpsinew}
 Given  sections %$X  $ be a metric space, $Y$ a topological space, $\pi :X \to Y$   a  quotient and
  $\phi, \psi :Y\to X$   
  %an intrinsically Lipschitz 
  of $\pi$. We say that   $\phi $ is {\em intrinsically $(L,M)$-quasi-isometric with respect to  $\psi$ at point $\hat x$}, with $L\geq1 , M\geq 0$ and $\hat x\in X$, if
\begin{enumerate}
\item $\hat x\in \psi(Y)\cap \phi (Y);$
\item $\phi  (Y) \cap  C_{\hat x,L}^{\psi} = \emptyset ,$
\end{enumerate}
where
$$ C_{\hat x,L}^{\psi} := \{x\in X \,:\,  d(x, \psi (\pi (x))) > L d(\hat x, \psi (\pi (x))) +M \}.   $$
\end{defi}

  \begin{rem} Definition~\ref{defwrtpsinew} can be rephrased as follows.
 A section $\phi  $ is intrinsically $(L,M)$-quasi-isometric with respect to  $\psi$ at point $\hat x$ if
 and only if 
 there is $\hat y\in Y$ such that  $\hat x= \phi (\hat y)=\psi(\hat y)$ and
 %we have that $x\notin C_{x_0,L}^{\phi_0},$ i.e.,
\begin{equation*}\label{defintrlipnuova}
 d(x, \psi (\pi (x))) \leq L d(\hat x, \psi (\pi (\hat x)))+M, \quad \forall x \in \phi (Y), 
\end{equation*}
which equivalently means %for any $y\in Y$
%On the other hand, if $x_0=\phi _0(y_0)=\phi (y_0),$ then  \eqref{defintrlipnuova} is equivalent to ask for $x=\phi (y)$
\begin{equation}\label{equation28.0}
 d(\phi (y), \psi (y)) \leq L d(\psi(\hat y), \psi (y))+M ,\qquad \forall y\in Y. 
\end{equation}
  \end{rem}

The proof of Proposition $\ref{linkintrinsicocneelip}$ is an immediately consequence of the following result.

 \begin{prop}\label{prop6apr}
   Let $X  $ be a metric space, $Y$ a topological space, and $\pi :X \to Y$   a  quotient map. Let $L\geq1, M \geq 0$ and $y_0\in Y$. Assume $\phi_0:Y\to X$ is an intrinsically $(L,M)$-quasi-isometric section of $\pi$. Let $\phi :Y\to X$ be a section of $\pi$
   such that 
   $x_0:=\phi (y_0)=\phi _0(y_0).$
   Then the following are equivalent:
   \begin{enumerate}
\item  For some $L_1\geq1$ and $M_1\geq 0$, $\phi $ is intrinsically $(L_1,M_1)$-quasi-isometric with respect to  $\phi_0$ at   $x_0;$
\item  For some $L_2\geq1$  and $M_2\geq 0$, $\phi $ satisfies   %is intrinsically $L_2$-Lipschitz at point $q=\phi (y)=\phi _0(y),$ i.e., 
\begin{equation}\label{equation3nov2021}
d(x_0,\phi (y)) \leq L_2d(x_0, \pi^{-1} (y))+M_2, \quad \forall y\in Y.
\end{equation}

\end{enumerate}
   Moreover, the constants $L_1$ and $L_2 $ are quantitatively related in terms of $L$ and  $M_1$ and $M_2 $ are quantitatively related in terms of $M$.
   \end{prop}

    \begin{proof}  

$(1) \Rightarrow (2).$ For every $y\in Y,$ it follows that
     \begin{equation*}
\begin{aligned}
d(\phi (y), x_0) & \leq d(\phi (y), \phi _0(y))  + d(\phi _0(y),x_0)  \\
& \leq (L_1+1)d(\phi _0(y), x_0)+M_1  \\
& \leq L(L_1+1)d(x_0, \pi^{-1} (y)) +M_1+M ,  \\
%& \leq L(L_1+1)d(x_0, \pi^{-1} (y_1)),  \\
\end{aligned}
\end{equation*}
   where in the first inequality we used the triangle inequality, and in the second one the intrinsic quasi-isometric property of $\phi.$ Then, in the third inequality we used the intrinsic quasi-isometric property of $\phi_0.$ %and, finally, in the last one we used \eqref{equation3nov2021}. %the fact that $\pi$ is a Lipschitz-quotient map.
      
      $(2) \Rightarrow (1).$   For every $y\in Y,$ we have that
\begin{equation*}
\begin{aligned}
d(\phi (y), \phi _0(y)) & \leq d(\phi (y),x_0) +  d( x_0, \phi _0 (y)) \\
%& \leq L _2 d_X(  \phi (y)  ,  \pi ^{-1} (y_1) )+ L d_X(  \phi _0(y)  ,  \pi ^{-1} (y_1) )\\
& \leq (L_2+1)  d(  \phi _0 (y)  ,  x_0)+M_2 ,
\end{aligned}
\end{equation*}
  where in the first equality we used the triangle inequality, and in the second one we used \eqref{equation3nov2021} and the fact that $\phi_0 (y)$ belongs to the fiber $\pi ^{-1} (y)$.  %In the third inequality we used the intrinsic Lipschitz property of $\phi_0$, and in the last equality we used the submetry of $\pi.$
   \end{proof}

   \begin{rem} 
 Notice that if $\phi_0$ in Proposition \ref{prop6apr} is $L$-intrinsically Lipschitz (i.e., $M=0$), then  $M_1=M_2.$
\end{rem}

\subsection{Convex set}
In this section we show that a class of intrinsically quasi-isometric sections with respect to another one is a convex set. We begin considering the following set.   
   \begin{defi}[Intrinsic quasi-isometric set with respect to $\psi$]\label{defwrtpsinew.9apr.23} Let  $\psi: Y \to X$ a section of $\pi$.  We define the set  of all  intrinsically quasi-isometric  sections  of $\pi$ with respect to  $\psi$ at point $\hat x$ as
\begin{equation*}
\begin{aligned}
I_{\psi , \hat x} & :=\{ \phi :Y\to X \mbox{ a section of $\pi$}  \, :\, \phi \mbox{ is intrinsically $(L,M)$-quasi-isometric w.r.t. $\psi$} \\
& \quad \quad \mbox{at point $\hat x$  for some $ L\geq 1$ and $M \geq 0$} \}.
 \end{aligned}
\end{equation*}
  \end{defi}

    \begin{prop}\label{propLeibnitz formula for slope} 
Let $\pi :X \to Y$ be a linear and quotient map with $X$ a normed space and $Y$ a topological space. Assume also that $\psi: Y \to X$ a section of $\pi$ and $\hat x \in \psi (Y).$ Then, the set $I_{\psi , \hat x} $ is a convex set.
\end{prop}

 \begin{proof} Let $\phi , \eta \in I _{\psi , \hat x} $ and let $t\in [0,1].$ We want to show that $$w := t\phi + (1-t)\eta \in I _{\psi , \hat x}.$$ 
Notice that by linearity of $\pi$ it holds $$\pi(w(y))=\pi (t\phi (y)+ (1-t)\eta (y)) =t\pi(\phi (y))+(1-t) \pi(\eta (y))=y,$$ for any $y\in Y$, i.e., $w$ is a section of $\pi.$ Moreover, $w (\bar y)= \phi (\bar y)=\eta (\bar y)=\hat x$ for some $\bar y \in Y$ and for every $y \in Y$ we have
\begin{equation*}
\begin{aligned}
\|w (y)-\psi(y)\|  & =\| t(\phi (y)- \psi ( y)) + (1-t) (\eta (y)-\psi( y))\|,
\end{aligned}
\end{equation*}
and so
\begin{equation*}
\begin{aligned}
\|w (y)-\psi (y)\| & \leq t \|\phi (y)-\psi ( y)\|+ (1-t) \|\eta (y)-\psi(y)\|.
\end{aligned}
\end{equation*}
Hence, 
 \begin{equation*}
\begin{aligned}
d(w (y),\psi  (y)) & \leq tL_\phi  d(\psi(\bar y), \psi (y))+M_\phi + (1-t)L_\eta  d(\psi(\bar y), \psi (y))+M_\eta\\ & = [t(L_\phi -L_\eta) + L_\eta ]  d(\psi(\bar y), \psi (y))+(M_\phi + M_\eta ),\\
\end{aligned}
\end{equation*}
% \begin{equation*}
%\begin{aligned}
%\frac{d(w (y),w (\bar y))}{d(w (\bar y), \pi ^{-1} (y)) ^\alpha}& \leq  t \frac{   d(\phi (y), \phi (\bar y))}{d(w (\bar y), \pi ^{-1} (y)) ^\alpha }  + 
%(1-t)\frac{ d(\eta (y),\eta (\bar y))}{d(w (\bar y), \pi ^{-1} (y)) ^\alpha }   \\
%\end{aligned}
%\end{equation*}
%where we omit the term $(f(y)-f(\bar y))(1-f( y)) d(\psi (y),\psi (\bar y)) / D$ because when we will take to the limit these terms are zero using the facts that $f\in ILS_{loc}(Z,\R)$ and $\phi (\bar y) = \psi (\bar y).$
for every $y\in Y,$ as desired.
\end{proof}

\subsection{Vector space: version 1}
In this section we show that a 'large' class of intrinsically quasi-isometric sections is a vector space over $\R$ or $\C$ defined as follows.    Notice that  it is no possible to obtain that $I_{\psi , \hat x }$ is a vector space for $\R$ since the simply observation that if $\psi (\hat y) = \hat x$ then $\psi (\hat y ) + \psi (\hat y )  \ne \hat x.$ On the other hand, we have the following result.

   \begin{theorem}\label{theorem24} Let $\pi :X \to Y$ is a linear and quotient map from a normed space $X$ to a topological space $Y.$ Assume also that $\psi :Y \to X$ is a section of $\pi$  and $\{\lambda \hat x\, :\, \lambda \in \R^+\} \subset X$ with $\hat x \in \psi (Y).$
   
Then, for any $\alpha \in (0,1],$ the set $\bigcup _{\lambda \in \R^+} I_{\lambda \psi , \lambda \hat x}  \cup \{ 0\}$ is a vector space over $\R $ or $\C .$
       \end{theorem}

\begin{proof} %of Theorem \ref{theorem24}. 
Let $\phi , \eta \in \bigcup_{\lambda \in \R ^+} I_{\lambda \psi , \lambda \hat x}$ and $\beta  \in \R  -\{0\}.$ We want to show that  
\begin{enumerate}
\item $w:=\phi + \eta \in \bigcup_{\lambda \in \R ^+} I_{\lambda \psi , \lambda \hat x }.$
\item $\beta \phi \in \bigcup_{\lambda \in \R ^+} I_{\lambda \psi , \lambda \hat x }.$
\end{enumerate}

(1). If $\phi  \in I_{\delta_1 \psi , \delta _1 \hat x  }$ and $ \eta \in  I _{\delta _2 \psi , \delta _2 \hat x }$ for some $\delta _1, \delta _2 \in \R^+$ it holds $$w \in  I_{(\delta _1 + \delta _2) \psi , ( \delta _1 + \delta _2) \hat x}.$$ 
For simplicity, we choose $\phi , \eta \in  I_{\psi , \hat x  }$ and so it remains to prove $$w \in  I_{2 \psi , 2 \hat x }.$$

By linear property of $\pi$, $w$ is a section of $1/2 \pi .$ On the other hand, if $\psi (\bar y) = \hat x,$ then  $w(\bar y)= \phi (\bar y)+ \eta(\bar y) = 2 \psi (\bar y) \in X.$ Moreover, using \eqref{equation28.0}, we deduce 
\begin{equation*}
\begin{aligned}
\|w(y)- 2\psi (y)\| & = \| \phi ( y)+ \eta( y) - 2 \psi (y)\|\\
& \leq \|\phi ( y)- \psi (y)\| +  \|\eta ( y)- \psi (y)\|\\
& \leq 2L \|\psi (\bar y)- \psi (y)\|+2M \\
& = L \| 2 \psi (\bar y) -2 \psi ( y)\| +2M,
\end{aligned}
\end{equation*}
for any $y \in Y,$ as desired.

(2). If $\phi  \in I_{\delta_1 \psi , \delta _1 \hat x  }$ then in a similar way to the point (1) it is possible to deduce that $\beta \phi \in I_{\beta \delta_1 \psi , \beta \delta _1 \hat x  }.$
\end{proof}

\subsection{Vector space: version 2} In this section we present another vector space over $\R$ or $\C$ for a suitable class of quasi-isometric sections of a linear map. Here we do not ask the condition given by Definition \ref{defwrtpsinew.9apr.23} but we want that the distance between two points belong to the same fiber is finite.

   \begin{theorem}\label{theorem30} Let $\pi :X \to Y$ is a linear and quotient map from a normed space $X$ to a topological space $Y.$ Assume also that the distance between two points of the same fiber is bounded by $\ell <\infty.$ Then, the set of all quasi-isometric sections of $\pi$ joint with the zero map is a vector space  over $\R$ or $\C.$
       \end{theorem}

\begin{proof} Let $\lambda  \in \R  -\{0\}$ and $\phi , \psi $ be two quasi-isometric sections of $\pi$ with constants $(L_\phi , M_\phi)$ and $(L_\psi , M_\psi),$ respectively.  We want to show that  
\begin{enumerate}
\item $\phi + \psi $ is a quasi-isometric section of $\pi.$
\item $\lambda \phi $ is a quasi-isometric section of $\pi.$
\end{enumerate}
(1). The fact that $ \phi +\psi$ is a section of $1/2\pi$ follows from the linearity of $\pi.$ Moreover, for any $y_1, y_2 \in Y$  we consider $a \in 1/2\pi^{-1}(y_2)$ such that $d( \phi (y_1) +\psi (y_1), 1/2 \pi ^{-1} (y_2)) = d( \phi (y_1) +\psi (y_1), a)$ and so 
\begin{equation*}
\frac 1 2 a \in \pi ^{-1} (y_2).
\end{equation*}
As a consequence, by triangle inequality,
  \begin{equation*}
\begin{aligned}
\| \phi (y_1) +\psi (y_1)- \phi (y_2) - \psi (y_2) \| & \leq \| \phi (y_1) +\psi (y_1)- a \| +\| a- \phi (y_2) - \psi (y_2) \| \\
& \leq  d( \phi (y_1) +\psi (y_1), a)  +\| 1/2 a- \phi (y_2)\| +\| 1/2a - \psi (y_2) \| \\
& \leq d( \phi (y_1) +\psi (y_1), 1/2 \pi ^{-1} (y_2)) +2\ell ,
\end{aligned}
\end{equation*}
where in the last inequality we used that the distance between two points in the same fiber is bounded by $\ell.$

(2). The fact that $\lambda \phi $ is a section of $1/\lambda \pi$ is trivial using the linearity of $\pi$. On the other hand, for any $y_1, y_2 \in Y$
  \begin{equation*}
\begin{aligned}
\| \lambda \phi (y_1) -\lambda \phi (y_2) \| \leq L_\phi |\lambda | d( \phi (y_1),  \pi^{-1} (y_2) ) +M_\phi= L_\phi d(\lambda \phi (y_1),  (1/\lambda \pi)^{-1} (y_2)) + M_\phi,
\end{aligned}
\end{equation*}
i.e., the thesis holds. This fact follows by these observations:
\begin{enumerate}
\item  if $ d( \phi (y_1),  \pi^{-1} (y_2) )=  d( \phi (y_1), a)$ then $ |\lambda | d( \phi (y_1),  \pi^{-1} (y_2) )= \|\lambda \phi (y_1)-\lambda a\|.$
\item $\lambda a \in \pi ^{-1} (\lambda y).$
\item $ \pi ^{-1} (\lambda y)=  (1/\lambda \pi)^{-1} ( y).$
\end{enumerate}
The second point is true because using the linearity of $\pi$ we have that $\pi (\lambda a) = \lambda \pi(a)=\lambda y.$ Finally, the third point holds since 
\begin{equation*}
 \pi^{-1} (\lambda y) =\{ x \in X \,:\,  \pi (x) = \lambda y\} = \{ x \in X \,:\, 1/ \lambda \pi (x) =  y\}= (1/ \lambda \pi)^{-1}  (y),
\end{equation*}
as desired.

\end{proof}

\section{An equivalence relation}\label{An equivalence relation} In this section $X$ is a metric space, $Y$ a topological space and $\pi:X \to Y$ a quotient map (we do $not$ ask that $\pi$ is a linear map). 
 We stress that Definition~\ref{defwrtpsinew} does not induce  an equivalence relation, because of lack of symmetry in the right-hand side of \eqref{equation28.0}. As a consequence we must ask a stronger condition  in order to obtain  an equivalence relation.
\begin{defi}[Intrinsic quasi-isometric  with respect to  a section in strong sense]\label{defwrtpsinew.1}
 Given  sections   $\phi, \psi :Y\to X$   of $\pi$. We say that   $\phi $ is {\em intrinsically $(L,M)$-quasi-isometric with respect to  $\psi$ at point $\hat x$ in strong sense}, with $L\geq1, M \geq 0$ and $\hat x\in X$, if
\begin{enumerate}
\item $\hat x\in \psi(Y)\cap \phi (Y);$
\item it holds
\begin{equation}\label{equation11aprile}
 d(\phi (y), \psi (y)) \leq L \min\{  d(\psi(\hat y), \psi (y)),  d(\psi(\hat y), \phi (y))\} +M,\qquad \forall y\in Y. 
\end{equation}
\end{enumerate}
\end{defi}

  Now we are able to give the main theorem.
  
   \begin{theorem}\label{theorem26aprile}
   Let $\pi :X \to Y$ be a quotient map from a metric space $X$ to a topological space $Y.$ Assume also that $\psi :Y \to X$ is a  section of $\pi$ and $\hat x \in X.$ Then, being intrinsically quasi-isometric with respect to $\psi$ at point $\hat x$ in strong sense induces an equivalence relation. We will write the class of equivalence of $\psi$ at point $\hat x$ as
    \begin{equation*}
\begin{aligned}
[I_{\psi , \hat x}]& :=\{\phi :Y\to X \, \mbox{a section of $\pi$}  :\, \phi \mbox{ is intrinsically $(L, M)$-quasi-isometric with }\\ 
& \qquad \mbox{respect to $\psi$ at point $\hat x$ in strong sense, for some $L\geq 1, M \geq 0$} \}.
\end{aligned}
\end{equation*}
   
      \end{theorem}

An interesting observation is that, considering $I_{\psi , \hat x },$ the intrinsic constants $L$ and $M$ can be change but it is fundamental that the point $\hat x$ is a common one for the every section. %On the other hand, the intrinsic constant $\alpha$ is the same for any element belongs to $[H_{\psi , \hat x}]$ thanks to \cite[Corollary 2.7]{D22.1}.

   \begin{proof}
We need to show:
\begin{enumerate}
\item reflexive property;
\item symmetric property;
\item transitive property;
\end{enumerate}

   (1). It is trivial that $\phi \backsim \phi.$
   
   (2). If $\phi \backsim \psi,$ then $\psi \backsim \phi .$ This follows from \eqref{equation11aprile}.
   
   (3). We know that $\phi \backsim \psi$ and $\psi \backsim \eta . $ Hence, $\hat x= \phi (\hat y)= \psi (\hat y)= \eta (\hat y).$ Moreover,  by \eqref{equation11aprile}, it holds
   \begin{equation*}\label{equation11aprile.0}
   \begin{aligned}
 d(\phi (y), \psi (y)) & \leq L_1 \min\{  d(\psi(\hat y), \psi (y)),  d(\psi(\hat y), \phi (y))\} +M_\phi,\\
  d(\psi (y), \eta (y)) & \leq L_2 \min\{  d(\eta(\hat y), \eta (y)),  d(\eta(\hat y), \psi (y))\} +M_\eta,\\
  \end{aligned}
\end{equation*}
for any $ y\in Y$ and, consequently, if $\tilde L = \min\{ L_1, L_2\},$ then
   \begin{equation}\label{equation11aprile.0}
   \begin{aligned}
 d(\phi (y), \psi (y)) & \leq  d(\phi (y), \psi (y)) +  d(\psi (y), \eta (y)) \\
  & \leq \tilde L \min\{  d(\eta(\hat y), \eta (y)),  d(\psi(\hat y), \phi (y))\} +M_\phi + M_\eta ,\\
   & = \tilde L \min\{  d(\eta(\hat y), \eta (y)),  d(\eta(\hat y), \phi (y))\} +M_\phi + M_\eta,\\
  \end{aligned}
\end{equation}
for any $ y\in Y.$ This means that $\phi \backsim \eta ,$ as desired.
      \end{proof}

\section{Proof of  Ahlfors-David regularity}\label{theoremAhlforsNEW} 

We need to a preliminary result.
 \begin{lem}\label{pseudodistance}
 Let $X$ be a metric  space, $Y$   a topological space, and $\pi:X\to Y$ a quotient map. If $\phi :Y \to X$ is an intrinsically $(L,M)$-quasi-isometric section of $\pi$ with $L\geq1$ and $M \geq 0,$ then it holds
		\begin{equation}\label{inclusionepalle}
\pi \left(B\left(p,\frac r L\right) \right) \subset \pi ( B(p,r +M) \cap \phi (Y)) \subset \pi (B(p,r+M)), \quad \forall p\in \phi (Y), \forall r>0.
\end{equation}
		
  \end{lem}

  \begin{proof}
%(i). 
 Regarding the first inclusion, fix $p\in \phi (Y), r>0$ and $q\in B(p, \frac r L).$  We need to show that $\pi (q) \in \pi (\phi (Y) \cap B(p,r +M)).$ Actually, it is enough to prove that 
\begin{equation}\label{equation2.6}
\phi (\pi (q)) \in B(p,r+M),
\end{equation}
 because if we take $g:= \phi (\pi (q)),$ then $g\in \phi (Y)$ and 
 \begin{equation*}
 \pi (g)= \pi (\phi (\pi (q))) =\pi (q) \in \pi (\phi (Y) \cap B(p,r +M)).
\end{equation*}
 
 Hence using the intrinsic quasi-isometric property of $\phi$ and the fact that $q$ and $ g$ belong to the same fiber because $ \pi (g) =\pi (q),$ we have that for any $p, q, g \in \phi (Y)$ with $g= \phi (\pi (q)),$
 \begin{equation}
d(p,g) \leq L d(p, \pi^{-1} (\pi (g)) )+M =  L d(p, \pi^{-1} (\pi (q)) )+M \leq L d(p,q) +M < r+M,
\end{equation}
i.e., \eqref{equation2.6} holds, as desired.  Finally, the second inclusion in \eqref{inclusionepalle} follows immediately noting that $\pi (\phi (Y))=Y$ because $\phi$ is a section and the proof is complete. %$\phi (Y)\cap B(p,r) \subset B(p,r),$ $\phi (Y)\cap B(p,r) \subset \phi (Y)$
  \end{proof}

  Now we are able to prove Theorem $\ref{thm2}$.
  \begin{proof}[Proof of Theorem $\ref{thm2}$]
   Let $\phi$ and $\psi$ intrinsically $(L,M)$-quasi-isometric sections, with $L\geq1$ and $M \geq 0$.
  Fix $y\in Y.$  By Ahlfors-David regularity of $\phi (y),$ we know that there are $c_1,c_2, r_0>0$ such that 
     \begin{equation}\label{AhlforsNEW0}
{c_1} r^Q \leq  \phi _* \mu \big( B(\phi (y),r) \cap \phi (Y)\big)  \leq c_2 r^Q,
\end{equation}
for all $ r >r_0.$ We would like to show that there is $c_3,c_4 >0$ such that
   \begin{equation}\label{AhlforsNEW127}
 c_4 (r-M)^Q \leq \psi _* \mu \big( B(\psi(y),r) \cap \psi (Y)\big)  \leq c_4 (r+M)^Q,
\end{equation}
  for every $r > r_0.$    
  We begin noticing that, by symmetry and  \eqref{Ahlfors27ott.112}
   \begin{equation}\label{Ahlfors27ott}
C^{-1} \mu (\pi (B(\psi(y),r))) \leq \mu (\pi (B(\phi(y),r))) \leq C \mu (\pi (B(\psi(y),r))).
%, \qquad \forall y\in Y.
\end{equation}
Moreover, 
\begin{equation}\label{serveperAhlfors27}
  \psi _* \mu \big( B(\psi(y),r) \cap \psi (Y)\big) =  \mu ( \psi^{-1} \big( B(\psi(y),r) \cap \psi (Y)\big) ) = \mu ( \pi \big( B(\psi(y),r) \cap \psi (Y)\big) ), 
\end{equation}
  and, consequently, %by Lemma~\ref{lemAhlfors}  and using \eqref{inclusionepalle} with $\psi$ in place of $\phi$, we deduce that
\begin{equation*}\label{ugualecarnot} 
\begin{aligned} \psi _* \mu \big( B(\psi(y),r) \cap \psi (Y)\big) & \geq  \mu (\pi (B(\psi (y),(r -M)/L))) \geq  C^{-1}   \mu (\pi (B(\phi (y),(r -M)/L))) \\
& \geq  C^{-1}   \mu (\pi (B(\phi (y),(r -M)/L) \cap \phi (Y)))\\
& =    C^{-1} \phi _* \mu \big( B(\phi(y),(r -M)/L ) \cap \phi (Y)\big) \\ & \geq   c_1C^{-1}  L^{-Q}  (r-M)^Q
\end{aligned}
\end{equation*}
where in the first inequality we used  the first inclusion of \eqref{inclusionepalle}  with $\psi$ in place of $\phi$, and in the second one we used \eqref{Ahlfors27ott}. In the  third  inequality we used the second inclusion of \eqref{inclusionepalle} and in the fourth one we used  \eqref{serveperAhlfors27}  with $\phi$ in place of $\psi.$ Moreover, in a similar way we have that 
\begin{equation*}
\begin{aligned} 
\psi _* \mu \big( B(\psi(y),r) \cap \psi (Y)\big) & \leq  \mu (\pi (B(\psi (y),r))) \leq C  \mu (\pi (B(\phi (y),r)))\\
& \leq C  \mu (\pi (B(\phi (y), L(r+M)) \cap \phi (Y)))
\\& = C  \phi _* \mu \big( B(\phi(y),L(r+M)) \cap \phi (Y)\big) \\
& \leq  {c_2} C  L^{Q} (r+M)^Q.
\end{aligned}
\end{equation*}
%where in the third equality we used that $\mathbb W$ is a normal subgroup, and in the fourth equality we used (c) of \cref{prop:PropertiesOfIntrinsicTranslation}. Then, the functions $\widetilde f_q\circ\pi_{\sW}\circ L_q$ and $\widetilde f\circ \pi_{\sW}$ differ only by a left translation of the element $q_{\mathbb L}$. Thus, in exponential coordinates, they are $\mathbb R^k$-valued functions that differ by the fixed Euclidean translation of the $\mathbb R^k$-vector corresponding to $q_{\mathbb L}$. This last observation comes from the fact that, in exponential coordinates, the operation of the group restricted to $\mathbb L$ is the Euclidean sum, being $\mathbb L$ horizontal, see \cite[Proposition 2.1]{FSSC03a}. Finally, \eqref{eqn:3} holds true by the fact that, component by component, we are differentiating along a vector field two functions that differ by a fixed constant.
Hence, putting together the last two inequalities we have that \eqref{AhlforsNEW127} holds with ${c_3} = c_1C^{-1} L^{-Q}$ and $c_4 = {c_2} C  L^{Q}.$ 
  \end{proof}

 \bibliographystyle{alpha}
\bibliography{DDLD}

\end{document}